\documentclass{amsart}

\usepackage{amssymb}
\usepackage{amsmath}
\usepackage{amsthm}
\usepackage{float}
\usepackage{hyperref}
\usepackage{subcaption}
\usepackage{tikz}
\usepackage[table]{xcolor}

\usepackage{pinlabel}
\makeatletter
\def\topt#1{\strip@pt\dimexpr #1*65536/\number\dimexpr 1pt}
\makeatother

\numberwithin{equation}{section}

\theoremstyle{plain}
\newtheorem{theorem}[equation]{Theorem}
\newtheorem{proposition}[equation]{Proposition}
\newtheorem{lemma}[equation]{Lemma}
\newtheorem{conjecture}[equation]{Conjecture}
\newtheorem{corollary}[equation]{Corollary}

\theoremstyle{definition}

\newtheorem{definition}[equation]{Definition}

\theoremstyle{remark}
\newtheorem{remark}[equation]{Remark}
\newtheorem{example}[equation]{Example}
\newtheorem*{acknowledgements}{Acknowledgements}

\newcommand{\cs}{\ensuremath{\mathbin{\#}}}

\newcommand{\redcrossing}{\tikz[baseline, yshift=0.4em]{
  \draw (135:0.4em) -- (315:0.4em);
  \draw (45:0.4em) -- (225:0.4em);
  \draw[thick, red] (0, 0) circle [radius=0.4em];
}}

\newcommand{\bluecrossing}{\tikz[baseline, yshift=0.4em]{
  \draw (135:0.4em) -- (315:0.4em);
  \draw (45:0.4em) -- (225:0.4em);
  \draw[thick, blue] (0, 0) circle [radius=0.4em];
}}

\title{The Ozsv\'ath--Szab\'o tau-invariant of braided satellites}
\author{Alex Eldridge}
\address{School of Mathematics, Georgia Institute of Technology, Atlanta, GA 30332}
\email{alex.eldridge@gatech.edu}

\hypersetup{
  pdftitle=The Ozsv\'ath--Szab\'o tau-invariant of braided satellites
}

\begin{document}

\begin{abstract}
We give formulas for the $ \tau $ and $ \varepsilon $ concordance invariants of satellite knots whose patterns are braided, meaning they wind around the solid torus without reversing.
Our methods lead us to define the class of squeezed patterns, analogous to squeezed knots as defined by Feller--Lewark--Lobb.
We show that all braided patterns are squeezed, and we give a $ \tau $ formula for squeezed patterns as well.
Towards a conjecture of Hedden, we show that no squeezed pattern, and thus no braided pattern, with winding number at least 2 induces a homomorphism on the concordance group.
\end{abstract}

\maketitle

\section{Introduction}

The Ozsv\'ath--Szab\'o $ \tau $-invariant, defined in terms of knot Floer homology, is a well-studied concordance invariant, giving, for example, a reproof of the Milnor conjecture~\cite{ozsvath-szabo-knot-floer-homology-03}. The bordered Heegaard Floer homology suite, developed by Lipshitz, Ozsv\'ath, and Thurston~\cite{lipshitz-ozsvath-thurston-bordered-heegaard-floer-18} and, more recently, its immersed curve interpretation, developed by Hanselman, Rasmussen, and Watson~\cite{hanselman-rasmussen-watson-bordered-floer-homology-24} provide methods for studying the knot Floer homology of satellites. For many patterns, these methods have given formulas for the behavior of $ \tau $ under satelliting. These include Whitehead doubles~\cite{hedden-knot-floer-homology-07}, cables~\cite{hom-bordered-heegaard-floer-14}, the Mazur pattern~\cite{levine-nonsurjective-satellite-operators-16}, and 1-bridge braids~\cite{chen-hanselman-satellite-knots-immersed-23-preprint}, among many others~\cite{bodish-genus-fiberedness-t-24-preprint, patwardhan-xiao-generalized-mazur-patterns-24-preprint, chen-zemke-zhou-applications-lspace-satellite-25-preprint}. We give a formula for $ \tau $ of ``braided'' satellites---satellites with patterns that wind around the solid torus without reversing, depicted below.

\begin{figure}[H]
  \labellist
  \pinlabel {$ \beta $} [r] at \topt{20mm} \topt{12.5mm}
  \pinlabel {$ P_\beta $} [r] at \topt{56mm} \topt{12.5mm}
  \endlabellist
  \centering
  \includegraphics{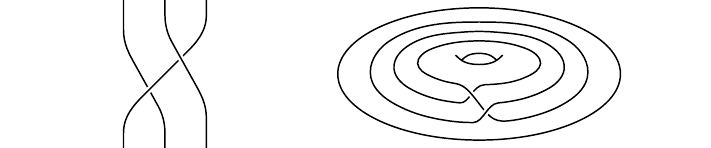}
  \caption{A braid $ \beta $ which closes to a knot and its correponding pattern $ P_\beta $, the closure of $ \beta $ in the solid torus.}
  \label{fig:braided-pattern}
\end{figure}

\begin{theorem}
  For an index-$ p $ braid $ \beta $ whose closure is a knot and a knot $ K \subset S^3 $, we have
  \[
    \tau(P_\beta(K)) = \begin{cases}
      p \tau(K) - \frac{1}{2}(p - 1) + \frac{1}{2}w(\beta) & \text{if } \varepsilon(K) = 1 \\
      p \tau(K) + \frac{1}{2}(p - 1) + \frac{1}{2}w(\beta) & \text{if } \varepsilon(K) = -1 \\
      \tau(P_\beta(U)) & \text{if } \varepsilon(K) = 0,
    \end{cases}
  \]
  where $ w(\beta) $ is the writhe of $ \beta $.
  \label{thm:tau-of-braided-satellites}
\end{theorem}

Hedden and Raoux gave an upper and lower bound on $ \tau(P_\beta(K)) $ which apply in a more general setting~\cite{hedden-raoux-4dimensional-rational-genus-23-preprint}. They allow $ K $ to be a framed, rationally null-homologous knot in a general 3-manifold and the closure of $ \beta $ to have more than one component, and they use an extension of $ \tau $ to rationally null-homologous links in general 3-manifolds. Specializing to our setting, their bounds become
\[
  p \tau(K) - \frac{1}{2}(p - 1) + \frac{1}{2} w(\beta)
  \leq \tau(P_\beta(K))
  \leq p \tau(K) + \frac{1}{2}(p - 1) + \frac{1}{2} w(\beta).
\]
Our formula determines $ \tau(P_\beta(K)) $ exactly, splitting into cases depending on Hom's $ \varepsilon $ concordance invariant, also defined using knot Floer homology~\cite{hom-bordered-heegaard-floer-14}. We give a formula for the $ \varepsilon $ invariant of braided satellites as well.

\begin{theorem}
  For a braid $ \beta $ whose closure is a knot and a knot $ K $, we have 
  \[
    \varepsilon(P_\beta(K)) = \begin{cases}
      \varepsilon(K) & \text{if } \varepsilon(K) \neq 0 \\
      \varepsilon(P_\beta(U)) & \text{if } \varepsilon(K) = 0.
    \end{cases}
  \]
  \label{thm:epsilon-of-braided-satellites}
\end{theorem}

These results are a generalization of Hom's formulas for $ \tau $ and $ \varepsilon $ of cables~\cite{hom-bordered-heegaard-floer-14}, and they recover Chen and Hanselman's formulas for $ \tau $ and $ \varepsilon $ of satellites with 1\nobreakdash-bridge braid patterns~\cite{chen-hanselman-satellite-knots-immersed-23-preprint}.

The reader may be interested to know that our proofs do not use bordered Heegaard Floer homology directly, instead relying on formal properties of $ \tau $. In~\cite{feller-lewark-lobb-squeezed-knots-24}, Feller, Lewark, and Lobb used $ \tau $'s formal properties to show that for a \textit{squeezed} knot $ K $, that is, one appearing as a slice of a minimal-genus cobordism from a negative torus knot to a positive one, $ \tau(K) $ is determined by how far along in that cobordism the knot appears. We similarly define the class of squeezed patterns, which we show includes all braided patterns.

\begin{definition}
  \label{def:squeezed-pattern}
  A pattern $ P \subset S^1 \times D^2 $ is \textit{squeezed} if there exist positive coprime integers $ p, q $ and a pair of cobordisms $ \Sigma^-\colon C_{p, -q} \to P $ and $ \Sigma^+\colon P \to C_{p, q} $ both in $ (S^1 \times D^2) \times [0, 1] $, such that 
  \[
    g(\Sigma^-) + g(\Sigma^+) = (p - 1) q.
  \]
\end{definition}

Equivalently, a squeezed pattern must be a slice of a minimal-genus cobordism from $ C_{p, -q} $ to $ C_{p, q} $. We give a $ \tau $ formula for squeezed patterns, which we use to prove Theorem~\ref{thm:tau-of-braided-satellites}.

\begin{theorem}
  \label{thm:tau-of-squeezed-pattern-satellite}
  If $ P $ is a pattern squeezed by cobordisms $ \Sigma^-\colon C_{p, -q} \to P $ and $ \Sigma^+\colon P \to C_{p, q} $, then for any knot $ K $, we have
  \[
    \tau(P(K)) = \begin{cases}
      \tau(K_{p, -q}) + g(\Sigma^-) = \tau(K_{p, q}) - g(\Sigma^+) & \text{if } \varepsilon(K) \neq 0 \\
      \tau(P(U)) & \text{if } \varepsilon(K) = 0.
    \end{cases}
  \]
\end{theorem}

\begin{figure}
  \centering
  \labellist
  \pinlabel {$ [K] \mapsto [U] $} [B] at \topt{20mm} \topt{2mm}
  \pinlabel {$ [K] \mapsto [K] $} [B] at \topt{60mm} \topt{2mm}
  \pinlabel {$ [K] \mapsto [K^r] $} [B] at \topt{100mm} \topt{2mm}
  \endlabellist
  \includegraphics{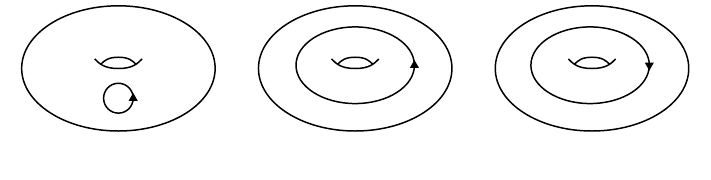}
  \caption{Three patterns and their induced maps on $ \mathcal{C} $, which are conjectured to be the only homomorphisms induced by satellites.}
  \label{fig:trivial-homomorphisms}
\end{figure}

These $ \tau $ formulas can be applied to a conjecture of Hedden. Satelliting by a pattern $ P $ induces a map $ [P]\colon \mathcal{C} \to \mathcal{C} $ on the concordance group given by $ [K] \mapsto [P(K)] $. One might ask how these maps interact with the group operation on $ \mathcal{C} $. In particular, when is this map $ [P] $ a homomorphism? That is, for which patterns $ P $ is $ P(K \cs K') $ always concordant to $ P(K) \cs P(K') $? This is clearly true for the three patterns in Figure~\ref{fig:trivial-homomorphisms}. Hedden conjectured that the homomorphisms induced by these three patterns are the only ones induced by satelliting.

\begin{conjecture}[{\cite[Conjecture 1]{hedden-pinzon-caicedo-satellites-infinite-rank-21}}]
  \label{conj:satellites-not-homomorphisms}
  For a pattern $ P $, if $ [P] $ is a homomorphism, then it is one of the maps $ [K] \mapsto [U] $, $ [K] \mapsto [K] $, or $ [K] \mapsto [K^r] $.
\end{conjecture}

Much partial progress has been made in proving this conjecture. In~\cite{lidman-miller-pinzon-caicedo-linking-number-obstructions-24}, it was shown that patterns satisfying a technical condition on linking numbers in branched covers do not induce homomorphisms, and in~\cite{johanningsmeier-kim-miller-partial-resolution-heddens-25}, this condition was shown to hold for all patterns with a winding number that is even but not divisible by 8.

In~\cite{miller-homomorphism-obstructions-satellite-23}, Miller gave an obstruction for a pattern $ P $ inducing a homomorphism in terms of $ \tau(P(K)) $. From this obstruction and Theorem~\ref{thm:tau-of-squeezed-pattern-satellite}, we can deduce that squeezed patterns with winding number at least 2, as well as patterns sufficiently close to those squeezed patterns, do not induce homomorphisms on concordance.

\begin{corollary}
  \label{cor:squeezed-not-homomorphism}
  A squeezed pattern $ P $ with winding number $ p \geq 2 $ does not induce a homomorphism on the concordance group. In fact, if a pattern $ Q $ cobounds a cobordism with $ P $ in $ (S^1 \times D^2) \times [0, 1] $ with genus strictly less than $ \left\lceil \frac{p - 1}{2} \right\rceil $, then $ Q $ does not induce a homomorphism either.
\end{corollary}

This paper came out of an attempt to obstruct some of the patterns known not to satisfy the condition in \cite{lidman-miller-pinzon-caicedo-linking-number-obstructions-24} from inducing a homomorphism. These are the ``alternating cable'' patterns~\cite[Figure~8]{lidman-miller-pinzon-caicedo-linking-number-obstructions-24} (also alluded to in~\cite[Example~5.12]{miller-homomorphism-obstructions-satellite-23}), and a particular index-8 braid~\cite[Figure~4]{johanningsmeier-kim-miller-partial-resolution-heddens-25}. The pattern depicted in Figure~\ref{fig:braided-pattern} is the $ (3, 1) $-alternating cable. All of these patterns are braided, so we show they do not induce homomorphisms.

In the last section we give examples of non-braided squeezed patterns and non-squeezed patterns. We also give a method for constructing non-braided squeezed patterns with any winding number at least 2.

\begin{acknowledgements}
We would like to thank Jennifer Hom, whose guidance and encouragement has been invaluable, and who, along with Neda Bagherifard and Sally Collins, invited us to speak about this project as a work in progress at the 2026 AMS Spring Southeastern Sectional Meeting, which helped to crystallize the ideas in this paper. We also thank Matt Hedden, Tye Lidman, and JungHwan Park for helpful comments.
\end{acknowledgements}

\section{Squeezed patterns}
\label{sec:squeezed-patterns}

In this paper, we require that cobordisms, both in $ S^3 \times [0, 1] $ and in $ (S^1 \times D^2) \times [0, 1] $, are smooth, compact, connected, and oriented.

We note a few facts that follow quickly from the definition of a squeezed pattern. First, winding number is an invariant of cobordisms in $ (S^1 \times D^2) \times [0, 1] $, so the squeezing cable patterns must have $ p $ equal to the winding number of the squeezed pattern. When the winding number is one, all cable patterns are isotopic to the core of the solid torus, so the minimal genus of a cobordism between any two such patterns is zero. Thus for winding-number-one patterns, being squeezed is the same as being concordant to the core.

Being squeezed is preserved under concordance. For a squeezed pattern $ P $ and a concordance $ \Sigma\colon P \to P' $, gluing $ \Sigma $ and $ \bar\Sigma $ to squeezing cobordisms for $ P $ give a squeezing of $ P' $.

To show a pattern is squeezed, it suffices to show that, for integers $ q < r $, it is a slice of a cobordism from $ C_{p, q} $ to $ C_{p, r} $ with Euler characteristic $ -(p - 1)(r - q) $. Such a cobordism can always be extended to a squeezing cobordism as in Definition~\ref{def:squeezed-pattern} by repeatedly gluing on the obvious cobordism from $ C_{p, q} $ to $ C_{p, q+1} $ given by attaching $ p - 1 $ bands. Note that $ q $ and $ r $ need not be coprime to $ p $, and may even have the same sign.

\subsection{The \texorpdfstring{$ \tau $}{τ} invariant and its formal properties}

Using knot Floer homology, Ozsv\'ath and Szab\'o defined the concordance invariant $ \tau $, and showed that it has the following properties.

\begin{theorem}[\cite{ozsvath-szabo-knot-floer-homology-03}]
  \label{thm:tau-is-slice-torus}
  The $ \tau $ invariant satisfies:
  \begin{enumerate}
    \item $ \tau(K \# J) = \tau(K) + \tau(J) $ for all knots $ K $, $ J $.
    \item $ \tau(K) \leq g_4(K) $ for all knots $ K $.
    \item $ \tau(T_{p, q}) = g_4(T_{p, q}) = \frac{1}{2}(p - 1)(q - 1) $ for $ p, q > 0 $.
  \end{enumerate}
\end{theorem}

Shortly afterwards, Livingston~\cite{livingston-computations-ozsvath-szabo-04} recognized that these properties alone are enough to calculate the value of $ \tau $ for many knots by getting bounds from cobordisms. If $ \Sigma\colon K_0 \to K_1 $ is a genus-$ g $ cobordism, then $ K_1 \cs -K_0 $ bounds a genus-$ g $ surface, and thus 
\[
  \tau(K_1) - \tau(K_0) = \tau(K_1 \cs - K_0) \leq g_4(K_1 \cs - K_0) \leq g(\Sigma).
\]
So, a cobordism gives a bound on how far apart the values of $ \tau $ can be on the ends, and knowing $ \tau $ on one end gives a bound for $ \tau $ on the other.

Since then, many other concordance invariants that also satisfy the properties in Theorem~\ref{thm:tau-is-slice-torus} have been discovered. Such invariants are called \textit{slice-torus}~\cite{lewark-rasmussens-spectral-sequences-14}. Feller, Lewark, and Lobb expanded Livingston's arguments to show that the value of any slice-torus invariant can be determined using only its formal properties for a large class of knots called \textit{squeezed} knots~\cite{feller-lewark-lobb-squeezed-knots-24}. Squeezed knots are knots which appear as a slice of a minimal-genus cobordism from a negative torus knot to a positive one.

\begin{remark}
The reader ought to be warned that for a pattern $ P $ squeezed between $ C_{p, -q} $ and $ C_{p, q} $, the knot $ P(U) $ is not in general squeezed between $ T_{p, -q} $ and $ T_{p, q} $. The difference is due to the fact that the minimal genus of a cobordism from $ C_{p, -q} $ to $ C_{p, q} $ in $ (S^1 \times D^2) \times [0, 1] $  is $ (p - 1)q $, but the minimal genus of a cobordism from $ T_{p, -q} $ to $ T_{p, q} $ in $ S^3 \times [0, 1] $ is $ (p - 1)(q - 1) $. This relates to the $ \varepsilon(K) = 0 $ case in which our squeezing argument breaks down, discussed later. In some sense, what makes our results possible is that the restriction that our cobordisms be in $ (S^1 \times D^2) \times [0, 1] $ is offset by the fact that we get a bit more leeway with the genus of our cobordisms.
\end{remark}

By slicing the minimal-genus cobordism for a squeezed knot $ K $, we get two cobordisms with torus knots at one end and $ K $ at the other, which give an upper and lower bound for $ \tau(K) $, and these bounds coincide.
This is shown for a general slice-torus invariant in Lemma 3.5 of~\cite{feller-lewark-lobb-squeezed-knots-24}.
It turns out that what makes the bounds coincide is that the genus of the cobordism of which $ K $ is a slice is equal to the difference of the values of the slice-torus invariant on the ends.
As we are only interested in calculating $ \tau $ and not other slice-torus invariants, we will use a slight generalization of that lemma, where the knots on the ends need not be torus knots.

\begin{lemma}
  \label{lemma:squeezing-tau}
  For a knot $ K $ with cobordisms $ \Sigma^-\colon K^- \to K $ and $ \Sigma^+\colon K \to K^+ $ such that 
  \[
    g(\Sigma^-) + g(\Sigma^+) = \tau(K^+) - \tau(K^-),
  \]
  we have
  \[
    \tau(K) = \tau(K^-) + g(\Sigma^-) = \tau(K^+) - g(\Sigma^+).
  \]
\end{lemma}

The proof of this is almost identical to the corresponding proof in~\cite{feller-lewark-lobb-squeezed-knots-24}, but we reproduce it here for clarity.

\begin{proof}
By drilling out a band between $ K^- $ and $ K $ in $ \Sigma^- $, we get a surface in $ B^4 $ with genus $ g(\Sigma^-) $ and boundary $ K \cs -K^- $, so $ g_4(K \cs -K^-) \leq g(\Sigma^-) $. Thus, by the properties in Theorem~\ref{thm:tau-is-slice-torus}, $ \tau(K) - \tau(K^-) \leq g(\Sigma^-) $. The same process for $ \Sigma^+ $ gives the bound $ \tau(K^+) - \tau(K) \leq g(\Sigma^+) $. Putting these bounds together gives
\[
    \tau(K^+) - g(\Sigma^+) \leq \tau(K) \leq \tau(K^-) + g(\Sigma^-).
\]
By hypothesis, the lower and upper bounds are equal, giving the desired equalities.
\end{proof}

\subsection{The \texorpdfstring{$ \tau $}{τ} invariant and squeezed patterns}

In~\cite{hom-bordered-heegaard-floer-14}, Hom defined the $ \{-1, 0, 1\} $-valued concordance invariant $ \varepsilon $, and gave the following formula for $ \tau $ of the $ (p, q) $-cable of a knot $ K $ in terms of $ p $, $ q $, $ \tau(K) $, and $ \varepsilon(K) $:

\begin{theorem}[{\cite[Theorem 1]{hom-bordered-heegaard-floer-14}}]
  \label{thm:tau-of-cable}
  For a knot $ K $ and $ p > 0 $,
  \[
    \tau(K_{p, q}) = \begin{cases}
      p \tau(K) + \frac{1}{2}(p - 1)(q - 1) & \text{if } \varepsilon(K) = 1 \\
      p \tau(K) + \frac{1}{2}(p - 1)(q + 1) & \text{if } \varepsilon(K) = -1 \\
      \tau(T_{p, q}) & \text{if } \varepsilon(K) = 0.
    \end{cases}
  \]
\end{theorem}

When $ \varepsilon(K) \neq 0 $, the cabling formula tells us that the difference in the values of $ \tau $ on $ K_{p, -q} $ and $ K_{p, q} $ is equal to $ (p - 1) q $, the genus of a squeezing cobordism from $ C_{p, -q} $ to $ C_{p, q} $. So, we are able to prove Theorem~\ref{thm:tau-of-squeezed-pattern-satellite} in this case using Lemma~\ref{lemma:squeezing-tau}. 
When $ \varepsilon(K) = 0 $, the $ \tau $-difference between the cables is too small for Lemma~\ref{lemma:squeezing-tau} to apply. Fortunately, the following theorem tells us that this case is not very interesting.

\begin{theorem}[\cite{hom-bordered-heegaard-floer-14}]
  \label{thm:tau-of-epsilon-zero}
  For a knot $ K $ with $ \varepsilon(K) = 0 $, and any pattern $ P $, 
  \[
    \tau(P(K)) = \tau(P(U)),
  \]
  where $ U $ is the unknot.
\end{theorem}

So, a complete $ \tau $ formula for any squeezed pattern can be given by calculating $ \tau(P(U)) $ using some other method.

\begin{proof}[Proof of Theorem~\ref{thm:tau-of-squeezed-pattern-satellite}]
Suppose the pattern $ P $ is squeezed by the pair of cobordisms $ \Sigma^-\colon C_{p, -q} \to P $ and $ \Sigma^+\colon P \to C_{p, q} $. For a knot $ K $, the cobordisms of patterns induce cobordisms $ \Sigma^-_K\colon K_{p, -q} \to P(K) $ and $ \Sigma^+_K\colon P(K) \to K_{p, q} $, now both in $ S^3 \times [0, 1] $, such that $ g(\Sigma^\pm_K) = g(\Sigma^\pm) $. If $ \varepsilon(K) \neq 0 $, Theorem~\ref{thm:tau-of-cable}, the cabling formula, tells us that
\[
  \tau(K_{p, q}) - \tau(K_{p, -q}) = (p - 1)q = g(\Sigma^-_K) + g(\Sigma^+_K).
\]
Therefore, the hypotheses for Lemma~\ref{lemma:squeezing-tau} hold, and we have
\[
  \tau(P(K)) = \tau(K_{p, -q}) + g(\Sigma^-) = \tau(K_{p, q}) - g(\Sigma^+),
\]
as desired. The $ \varepsilon(K) = 0 $ case follows immediately from Theorem~\ref{thm:tau-of-epsilon-zero}.
\end{proof}

\section{Braided patterns}

\subsection{The \texorpdfstring{$ \tau $}{τ} invariant of braided satellites}

Consider the cobordism from $ C_{p, -q} $ to $ C_{p, q} $ in the solid torus given by first attaching a band that resolves each negative crossing, and then attaching another band that adds a positive crossing in its place. This cobordism consists of $ 2(p - 1)q $ band moves, and thus has genus $ (p - 1)q $. The fact that these bands can be attached in any order allows us to achieve any braid as a slice of a cobordism constructed in this way.

\begin{proof}[Proof of Theorem~\ref{thm:tau-of-braided-satellites}]
For an index-$ p $ braid $ \beta $, let $ k^+ $ be the number of positive crossings, and let $ k^- $ be the number of negative crossings. Also, let $ \beta^+ $ be the braid given by making all of the negative crossings in $ \beta $ positive, and $ \beta^- $ be the braid given by making all of the positive crossings in $ \beta $ negative. Pick an integer $ q $ that is coprime to $ p $ and large enough that a braid word for $ \beta^+ $ appears as a subsequence of the braid word $ (\sigma_1 \sigma_2 \cdots \sigma_{p - 1})^q $. Then, a braid word for $ \beta^- $ appears as a subsequence of $ (\sigma_1^{-1} \sigma_2^{-1} \cdots \sigma_{p - 1}^{-1})^q $ as well. We construct squeezing cobordisms as depicted in Figure~\ref{fig:braid-squeezing-cobordisms}.

\begin{figure}[b]
  \centering
  \labellist
  \footnotesize
  \pinlabel {$ C_{p, -q} $} at \topt{8.75mm} \topt{5mm}
  \pinlabel {$ \beta^- $} at \topt{35.875mm} \topt{5mm}
  \pinlabel {$ \beta $} at \topt{63mm} \topt{5mm}
  \pinlabel {$ \beta^+ $} at \topt{90.125mm} \topt{5mm}
  \pinlabel {$ C_{p, q} $} at \topt{117.25mm} \topt{5mm}
  \pinlabel \parbox{10em}{\centering resolve \\ all \redcrossing} [b] at \topt{22.312mm} \topt{26mm}
  \pinlabel \parbox{10em}{\centering change \\ \bluecrossing} [b] at \topt{49.437mm} \topt{26mm}
  \pinlabel \parbox{10em}{\centering change \\ all \bluecrossing} [b] at \topt{76.562mm} \topt{26mm}
  \pinlabel \parbox{10em}{\centering finish \\ adding \\ bands} [b] at \topt{103.687mm} \topt{26mm}
  \endlabellist
  \includegraphics{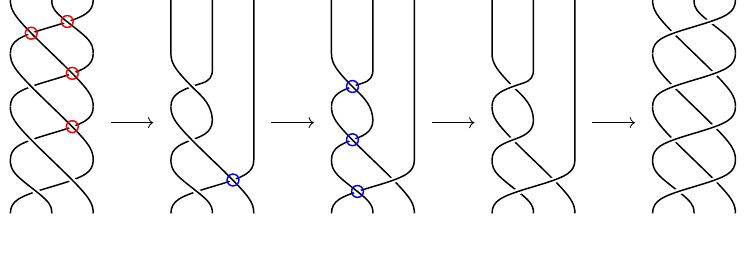}
  \caption{The construction of a genus-$ (p - 1)q $ cobordism from $ C_{p, -q} $ to $ C_{p, q} $ of which an arbitrary braid is a slice.}
  \label{fig:braid-squeezing-cobordisms}
\end{figure}
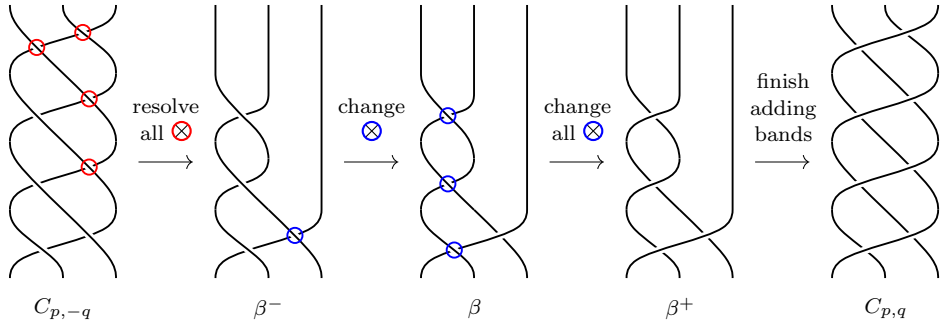

Construct the cobordism $ \Sigma^-\colon C_{p, -q} \to P_\beta $ with the following braid operations. The pattern $ C_{p, -q} $ corresponds to the braid $ (\sigma_1^{-1} \sigma_2^{-1} \cdots \sigma_{p - 1}^{-1})^q $. Delete braid generators until only the braid $ \beta^- $ remains. This involves deleting $ (p - 1) q - (k^+ + k^-) $ generators. Then, change $ k^+ $ crossings to turn $ \beta^- $ to $ \beta $. The resulting cobordism $ \Sigma^- $ has genus equal to half the number of band moves, so
\[
    g(\Sigma^-) = \frac{1}{2} \left( (p - 1) q + (k^+ - k^-) \right).
\]

Similarly, construct $ \Sigma^+\colon P_\beta \to C_{p, q} $ with the following braid operations. Turn $ \beta $ into $ \beta^+ $ with $ k^- $ crossing changes. Then, add $ (p - 1) q - (k^+ + k^-) $ generators to turn $ \beta^+ $ to $ (\sigma_1 \sigma_2 \cdots \sigma_{p - 1})^q $. So, 
\[
    g(\Sigma^+) = \frac{1}{2} \left( (p - 1) q - (k^+ - k^-) \right),
\]
and $ g(\Sigma^-) + g(\Sigma^+) = (p - 1) q $. Therefore, $ P_\beta $ is squeezed.

Note the writhe of the braid $ \beta $ is $ k^+ - k^- $, so $ g(\Sigma^-) = \frac{1}{2}\left( (p - 1) q + w(\beta) \right) $. Thus, using the $ \tau $ formulas in Theorem~\ref{thm:tau-of-squeezed-pattern-satellite} and Theorem~\ref{thm:tau-of-cable}, for $ K $ with $ \varepsilon(K) \neq 0 $, we have 
\[
\begin{aligned}
  \tau(P_\beta(K)) &= \tau(K_{p, -q}) + g(\Sigma^-) \\
  &= p\tau(K) + \frac{1}{2}(p - 1)(-q - \varepsilon(K)) + \frac{1}{2}\left( (p - 1) q + w(\beta) \right) \\
  &= p\tau(K) - \frac{1}{2}(p - 1)\varepsilon(K) + \frac{1}{2} w(\beta),
\end{aligned}
\]
as desired. The $ \varepsilon(K) = 0 $ case again follows immediately from Theorem~\ref{thm:tau-of-epsilon-zero}.
\end{proof}

\begin{remark}
For a knot $ K $ with $ \varepsilon(K) = 0 $, one may still get an upper and lower bound on $ \tau(P_\beta(K)) $ using the cobordisms $ \Sigma^\pm $ defined in the preceding proof. These bounds are 
\[
  -\frac{1}{2}(p - 1) + \frac{1}{2} w(\beta)
  \leq \tau(P_\beta(K))
  \leq \frac{1}{2}(p - 1) + \frac{1}{2} w(\beta).
\]
If $ \varepsilon(K) = 0 $, then we must have $ \tau(K) = 0 $ \cite[Proposition 3.6]{hom-bordered-heegaard-floer-14}, so these  are exactly the bounds on $ \tau(P_\beta(K)) $ given by Hedden and Raoux in \cite{hedden-raoux-4dimensional-rational-genus-23-preprint}.
\end{remark}

Theorem \ref{thm:tau-of-braided-satellites} can also be interpreted in terms of work of Van Cott~\cite{vancott-ozsvathszabo-rasmussen-invariants-10}. First, we describe a similar reinterpretation of the cabling formula in Theorem~\ref{thm:tau-of-cable}. Van Cott defines the function $ h(q) $ for a fixed winding number $ p $, knot $ K $, and slice-torus invariant $ \nu $ as 
\[
  h(q) = \nu(K_{p, q}) - \frac{1}{2}(p - 1) q,
\]
where $ q $ is an integer coprime to $ p $. Theorem 2 of~\cite{vancott-ozsvathszabo-rasmussen-invariants-10} says that for $ q > q' $, we have
\[
  -(p - 1) \leq h(q) - h(q') \leq 0.
\]
In particular, $ h $ is non-increasing. Hom's cabling formula---which itself relies on this work of Van Cott---pins down $ h $ exactly when the slice-torus invariant is $ \tau $.

Van Cott has a similar result for braided patterns. Fix an index-$ p $ braid $ \beta $, a knot $ K $, and a slice-torus invariant $ \nu $. Let $ \beta_r $ be the braid $ \beta $ with $ r $ full twists appended (in terms of braid generators, $ \beta_r = \beta (\sigma_1 \cdots \sigma_{p - 1})^{pr} $). Then define 
\[
  g(r) = \nu(P_{\beta_r}(K)) - \frac{1}{2}(p - 1) pr.
\]
Theorem 9 of~\cite{vancott-ozsvathszabo-rasmussen-invariants-10} says that for $ r > r' $, we have 
\[
  -(p - 1) \leq g(r) - g(r') \leq 0.
\]
For a fixed $ K $ with $ \varepsilon(K) \neq 0 $ and $ p $, and with $ \tau $ as our slice-torus invariant, let $ g_\tau $ be the corresponding $ g $ function. Then, our Theorem~\ref{thm:tau-of-braided-satellites} is equivalent to: 
\[
  g_\tau(r) = \begin{cases}
    p\tau(K) - \frac{1}{2}(p - 1) + \frac{1}{2} w(\beta) & \text{if } \varepsilon(K) = 1 \\
    p\tau(K) + \frac{1}{2}(p - 1) + \frac{1}{2} w(\beta) & \text{if } \varepsilon(K) = -1
  \end{cases}
\]
Our methods do not work in the $ \varepsilon(K) = 0 $ case, as by Theorem~\ref{thm:tau-of-epsilon-zero}, this is equivalent to finding $ \tau $ of the closure of an arbitrary braid.

\subsection{The \texorpdfstring{$ \varepsilon $}{ε} invariant of braided satellites}

Armed with a $ \tau $ formula for a pattern $ P $, if one can calculate the $ \tau $ invariant for cables of $ P(K) $, then one can check which case of the cabling formula agrees with that $ \tau $ value and deduce $ \varepsilon(P(K)) $. This technique is used in~\cite{hom-bordered-heegaard-floer-14} to prove the following theorem.

\begin{theorem}[{\cite[Theorem 2]{hom-bordered-heegaard-floer-14}}]
  \label{thm:epsilon-of-cable}
  If $ \varepsilon(K) \neq 0 $, then $ \varepsilon(K_{p, q}) = \varepsilon(K) $.
\end{theorem}

Likewise, we prove Theorem~\ref{thm:epsilon-of-braided-satellites} by calculating $ \tau((P_\beta(K))_{2, q}) $ two ways. First, we recall the definition of a composite pattern. Given two patterns $ P, Q \colon S^1 \to S^1 \times D^2 $, let $ i_Q \colon S^1 \times D^2 \to S^1 \times D^2 $ be an embedding of a Seifert-framed tubular neighborhood of $ Q $. Then, their composite pattern, denoted $ P \circ Q $, is the pattern given by the composition $ i_Q \circ P $. It is clear from the definition that $ (P \circ Q)(K) $ is isotopic to $ P(Q(K)) $.

Also, as before, our argument breaks down in the $ \varepsilon(K) = 0 $ case. We can again appeal to a theorem of Hom to fill in this case.

\begin{theorem}[\cite{hom-bordered-heegaard-floer-14}]
  \label{thm:epsilon-of-epsilon-zero}
  For a knot $ K $ with $ \varepsilon(K) = 0 $, and any pattern $ P $, 
  \[
    \varepsilon(P(K)) = \varepsilon(P(U)).
  \]
\end{theorem}

\begin{proof}[Proof of Theorem~\ref{thm:epsilon-of-braided-satellites}]
Suppose $ \varepsilon(K) \neq 0 $ and $ \beta $ is an index-$ p $ braid. Then, we can find $ \tau $ of the $ (2, q) $-cable of $ P_\beta(K) $ by applying Theorem~\ref{thm:tau-of-cable}:
\begin{equation}
  \tau((P_\beta(K))_{2, q}) = \begin{cases}
    2 \tau(P_\beta(K)) + \frac{1}{2}(q - \varepsilon(P_\beta(K))) & \text{if } \varepsilon(P_\beta(K)) \neq 0 \\
    \tau(T_{2, q}) & \text{if } \varepsilon(P_\beta(K)) = 0.
  \end{cases}
  \label{eq:composite-formula-1}
\end{equation}

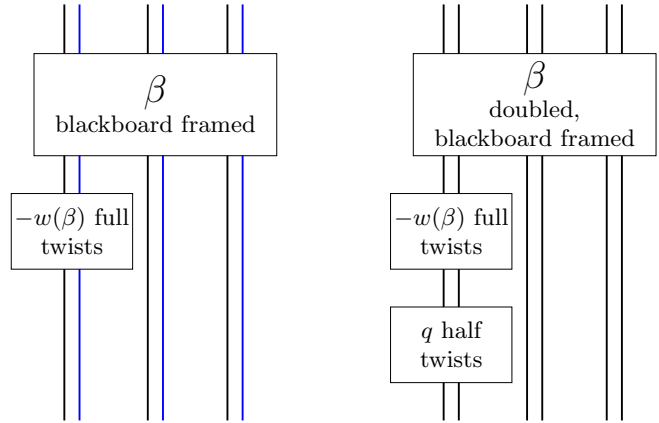
\begin{figure}[b]
  \centering
  \labellist
  \small
  \pinlabel \parbox{10em}{\centering {\LARGE $ \beta $} \\ blackboard framed} at \topt{20mm} \topt{41.75mm}
  \pinlabel \parbox{10em}{\centering $ -w(\beta) $ full \\ twists} at \topt{9mm} \topt{25mm}
  \pinlabel \parbox{10em}{\centering {\LARGE $ \beta $} \\ doubled, \\ blackboard framed} at \topt{70mm} \topt{41.75mm}
  \pinlabel \parbox{10em}{\centering $ -w(\beta) $ full \\ twists} at \topt{59mm} \topt{25mm}
  \pinlabel \parbox{10em}{\centering $ q $ half \\ twists} at \topt{59mm} \topt{10mm}
  \endlabellist
  \includegraphics{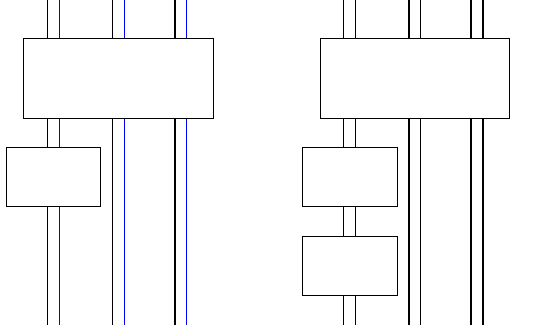}
  \caption{Left: a braid $ \beta $ with a Seifert-framed longitude marked in blue. Right: the braid $ \beta_{2, q} $ representing the $ (2, q) $-cable of the pattern $ P_\beta $.}
  \label{fig:cable-of-braid}
\end{figure}

Now, consider the composite pattern $ C_{2, q} \circ P_\beta $. This composite pattern is also braided, given by the braid $ \beta_{2, q} $ depicted in Figure~\ref{fig:cable-of-braid}, so $ (P_\beta(K))_{2, q} = P_{\beta_{2, q}}(K) $. Note that the braid $ \beta_{2, q} $ has index $ 2p $. Note also that each crossing in $ \beta $ becomes four crossings of the same sign in $ \beta_{2, q} $, so, together with the $ -w(\beta) $ full twists in the framing to account for the writhe and $ q $ half twists in the cable pattern, the total writhe of $ \beta_{2, q} $ is $ 2w(\beta) + q $. So, we can apply Theorem~\ref{thm:tau-of-braided-satellites} directly, to get the following:
\begin{equation}
  \tau(P_{\beta_{2, q}}(K)) = 2p \tau(K) - \frac{1}{2}(2p - 1)\varepsilon(K) + \frac{1}{2}(2 w(\beta) + q)
  \label{eq:composite-formula-2}
\end{equation}

These formulas tell us that $ \varepsilon(P_\beta(K)) $ cannot be $ 0 $. If it were, \eqref{eq:composite-formula-1} would say that $ \tau((P_\beta(K))_{2, -1}) $ and $ \tau((P_\beta(K))_{2, 1}) $ are equal, as $ T_{2, -1} $ and $ T_{2, 1} $ are both unknots, and \eqref{eq:composite-formula-2} would say they are not. Thus, the right-hand side of \eqref{eq:composite-formula-2} is equal to the $ \varepsilon(P_\beta(K)) \neq 0 $ case of \eqref{eq:composite-formula-1}, giving
\[
  2 \tau(P_\beta(K)) + \frac{1}{2}(q - \varepsilon(P_\beta(K)))
  = 2p \tau(K) - \frac{1}{2}(2p - 1)\varepsilon(K) + \frac{1}{2}(2 w(\beta) + q).
\]
After rewriting $ \tau(P_\beta(K)) $ using Theorem~\ref{thm:tau-of-braided-satellites}, all of the terms involving $ p $, $ q $, $ w(\beta) $, and $ \tau(K) $ cancel, leaving $ \varepsilon(P_\beta(K)) = \varepsilon(K) $, as desired. The $ \varepsilon(K) = 0 $ case follows immediately from Theorem~\ref{thm:epsilon-of-epsilon-zero}.
\end{proof}

\section{Homomorphism obstructions}

We use the following obstruction in terms of the $ \tau $ invariant, which says that if a pattern induces a homomorphism on $ \mathcal{C} $, its $ \tau $ formula cannot depend on $ \varepsilon(K) $.

\begin{proposition}[{\cite[Proposition 5.4]{miller-homomorphism-obstructions-satellite-23}}]
  \label{prop:tau-of-homom-is-mult-by-winding-number}
  If a pattern $ P $ induces a homomorphism on the concordance group, and $ p \geq 0 $ is the winding number of $ P $, then 
  \[
      \tau(P(K)) = p\tau(K).
  \]
\end{proposition}

A squeezed pattern has a $ \tau $ formula which differs from that of a cable pattern by some constant, and this does not look like the above $ \tau $ formula for winding number at least 2. This is made formal in the following proposition, from which Corollary~\ref{cor:squeezed-not-homomorphism} follows quickly.

\begin{proposition}
  \label{prop:squeezed-far-from-homomorphism}
  Suppose $ P $ is a pattern for which there exists $ p \geq 2 $, $ q $ coprime to $ p $, and a constant $ c $ such that for $ K $ with $ \varepsilon(K) \neq 0 $, 
  \[
      \tau(P(K)) = \tau(K_{p, q}) + c.
  \]
  Suppose also that $ Q $ is a pattern that induces a homomorphism on concordance. If $ \Sigma $ is a cobordism between them, then 
  \[
      g(\Sigma) \geq \left| \frac{1}{2}(p - 1)q + c \hspace{1pt}\right| + \frac{1}{2}(p - 1).
  \]
\end{proposition}

\begin{proof}[Proof of Corollary~\ref{cor:squeezed-not-homomorphism}]
By Theorem~\ref{thm:tau-of-squeezed-pattern-satellite}, the squeezed pattern $ P $ satisfies the hypotheses of Proposition~\ref{prop:squeezed-far-from-homomorphism}. If $ g(\Sigma) < \lfloor \frac{p - 1}{2} \rfloor $, then $ \Sigma $ does not satisfy the genus bound in Proposition~\ref{prop:squeezed-far-from-homomorphism}, so $ Q $ does not induce a homomorphism.
\end{proof}

\begin{remark}
One obvious obstruction to a pattern $ P $ inducing a homomorphism is that the knot $ P(U) $ must be slice. We call such patterns \textit{slice}. If a pattern $ P $ is not slice, we can make a slice pattern $ P^\# $ out of it by connected-summing on the knot $ -P(U) $ in the solid torus, as in the following figure.

\begin{figure}[H]
  \centering
  \labellist
  \pinlabel {$ P $} [B] at \topt{25mm} \topt{0mm}
  \pinlabel {$ P^\# $} [B] at \topt{85mm} \topt{0mm}
  \endlabellist
  \includegraphics{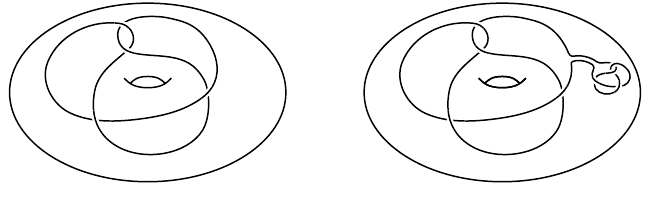}
  \caption{A pattern $ P $ and its corresponding slice pattern $ P^\# $.}
  \label{fig:make-pattern-slice}
\end{figure}

Or, in other words, the satellite $ P^\#(K) $ is $ P(K) \cs -P(U) $. So, when $ P $ is squeezed with lower squeezing cobordism $ \Sigma^-\colon C_{p, -q} \to P $, its corresponding slice pattern $ P^\# $ has the following partial $ \tau $ formula for $ K $ with $ \varepsilon(K) \neq 0 $:
\[
  \tau(P^\#(K)) = \tau(P(K)) - \tau(P(U)) = \tau(K_{p, -q}) + g(\Sigma^-) - \tau(P(U)).
\]
So, for $ \varepsilon(K) \neq 0 $, $ \tau(P^\#(K)) $ also differs from a $ \tau(K_{p, -q}) $ by a uniform constant. This is all to say that a squeezed pattern $ P $ might easily be seen not to induce a homomorphism if $ P(U) $ is not slice, but Proposition~\ref{prop:squeezed-far-from-homomorphism} tells us that $ P^\# $ does not induce a homomorphism either.
\end{remark}

It follows quickly from the slice-torus properties in Theorem~\ref{thm:tau-is-slice-torus} that the slice genus bound can be strengthened to $ |\tau(K)| \leq g_4(K) $. For a cobordism of knots $ \Sigma\colon K \to K' $ in $ S^3 \times [0, 1] $, the corresponding stronger genus bound is $ g(\Sigma) \geq |\tau(K') - \tau(K)| $. We use this stronger bound in the following proof.

\begin{proof}[Proof of Proposition~\ref{prop:squeezed-far-from-homomorphism}]
Pick a knot $ K^+ $ with $ \varepsilon(K^+) = 1 $. Then, using the formulas for $ \tau(P(K^+)) $ by hypothesis and $ \tau(Q(K^+)) $ by Proposition~\ref{prop:tau-of-homom-is-mult-by-winding-number}, we have
\[
  g(\Sigma)
  \geq |\tau(P(K^+)) - \tau(Q(K^+))|
  = \left| \frac{1}{2}(p - 1)(q - 1) + c \hspace{1pt}\right|.
\]
Similarly, for $ K^- $ with $ \varepsilon(K^-) = -1 $, we have 
\[
  g(\Sigma)
  \geq |\tau(P(K^-)) - \tau(Q(K^-))|
  = \left| \frac{1}{2}(p - 1)(q + 1) + c \hspace{1pt}\right|.
\]
Putting these bounds together gives 
\[
\begin{aligned}
  g(\Sigma)
  &\geq \max \left( \left| \frac{1}{2}(p - 1)(q - 1) + c \hspace{1pt}\right|, \left| \frac{1}{2}(p - 1)(q + 1) + c \hspace{1pt}\right| \right) \\
  &= \left| \frac{1}{2}(p - 1)q + c \hspace{1pt}\right| + \frac{1}{2}(p - 1),
\end{aligned}
\]
as desired.
\end{proof}

\section{Examples and non-examples}

The most straightforward obstruction to a pattern being squeezed is that its $ \tau $ formula must be of the form described in Theorem~\ref{thm:tau-of-squeezed-pattern-satellite}. A simple test for this is that for a squeezed pattern $ P $ with winding number $ p $, we must have 
\[
  \tau(P(T_{2,3})) - \tau(P(-T_{2,3})) = p + 1,
\]
using the fact that $ \varepsilon(\pm T_{2, 3}) = \pm 1 $~\cite[Proposition 3.6]{hom-bordered-heegaard-floer-14}.

\begin{example}
In~\cite{patwardhan-xiao-generalized-mazur-patterns-24-preprint}, Patwardhan and Xiao define a family of ``generalized Mazur patterns'' $ Q_{m, n} $ and give a $ \tau $ formula for them. For integers $ m > n > 1 $, the pattern $ Q_{m, n} $ has winding number $ m - n $, but applying the $ \tau $ formula gives
\[
  \tau(Q_{m, n}(T_{2,3})) - \tau(Q_{m, n}(-T_{2,3})) = m > m - n + 1.
\]
So, these patterns fail our test and give infinitely many examples of non-squeezed patterns for each positive winding number.
\end{example}

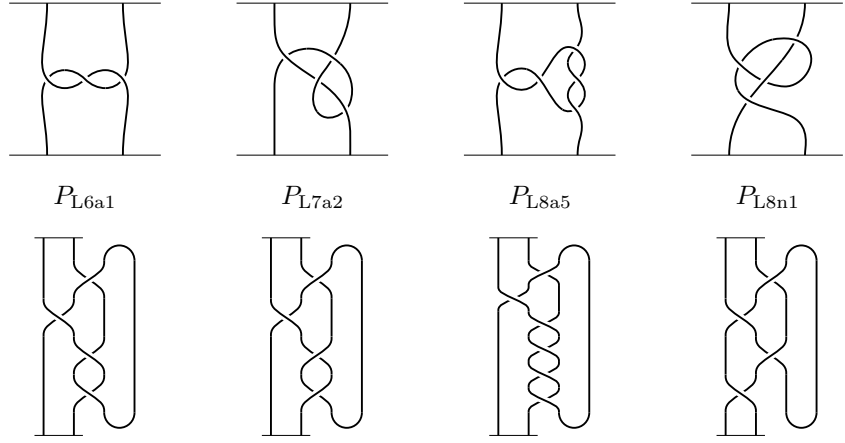
\begin{figure}
  \labellist
  \pinlabel {$ P_\text{L6a1} $} at \topt{15mm} \topt{32.5mm}
  \pinlabel {$ P_\text{L7a2} $} at \topt{45mm} \topt{32.5mm}
  \pinlabel {$ P_\text{L8a5} $} at \topt{75mm} \topt{32.5mm}
  \pinlabel {$ P_\text{L8n1} $} at \topt{105mm} \topt{32.5mm}
  \endlabellist
  \centering
  \includegraphics{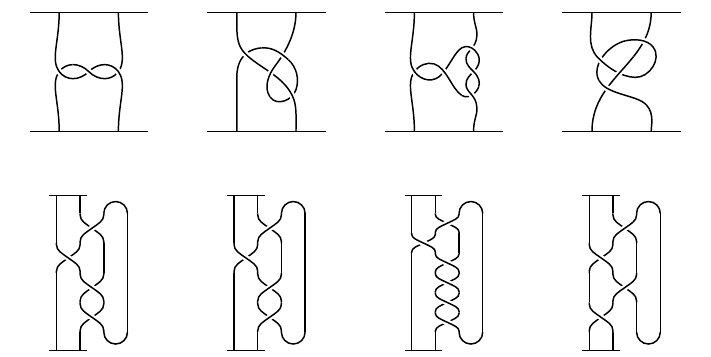}
  \caption{Four patterns corresponding to small crossing number links and their representations as braids with two strands running around the solid torus and one strand not. To save space, patterns are drawn in $ D^2 \times [0, 1] $ with the top and bottom edges identified.}
  \label{fig:small-patterns}
\end{figure}

A 2-component link in $ S^3 $ with a distinguished unknotted component corresponds to a pattern in the solid torus by removing a tubular neighborhood of that unknotted component. We consider such patterns coming from small crossing number links in the LinkInfo table~\cite{linkinfo}. For each of the links we discuss, there is either only one unknotted component, or an isotopy that exchanges the two unknotted components, so each link corresponds to a unique pattern, which we denote by its name in LinkInfo (e.g. $ P_\text{L6a1} $).

\begin{example}
The patterns corresponding to L6a1, L7a2, L8a5, and L8n1 are not braided, but they can be written as index-3 braids with only two of their strands traveling around the solid torus, as in Figure~\ref{fig:small-patterns}. Recall from the discussion at the beginning of Section~\ref{sec:squeezed-patterns} that to show these patterns are squeezed, it suffices to find them as a slice of a cobordism $ \Sigma\colon C_{2, q} \to C_{2, r} $ with $ \chi(\Sigma) = -(r - q) $. Using a computer search that applies braid relations, resolves crossings, and changes crossings, we find such cobordisms in Figure~\ref{fig:small-pattern-cobordisms}. The squeezing cobordisms for $ P_\text{L6a1} $ and $ P_\text{L8a5} $ easily generalize to an infinite family of squeezed patterns.
\end{example}

\begin{figure}
  \begin{subfigure}{\textwidth}
    \centering
    \labellist
    \footnotesize
    \pinlabel {$ -k $} at \topt{6mm} \topt{54.5mm}
    \pinlabel {$ -k $} at \topt{26mm} \topt{54.5mm}
    \pinlabel {$ -k $} at \topt{58mm} \topt{44.5mm}
    \pinlabel {$ -k $} at \topt{86mm} \topt{44.5mm}
    \pinlabel {$ -k $} at \topt{86mm} \topt{9.5mm}
    \pinlabel {isotopy} [B] at \topt{16mm} \topt{53mm}
    \pinlabel {isotopy} [B] at \topt{44mm} \topt{53mm}
    \pinlabel {change \bluecrossing} [B] at \topt{72mm} \topt{53mm}
    \pinlabel {\parbox{10em}{\centering resolve \\ all \redcrossing, \\ cap off \\ disk}} [b] at \topt{100mm} \topt{52mm}
    \pinlabel {isotopy} [tr] at \topt{75mm} \topt{35mm}
    \pinlabel {$ C_{2,-2k-3} $} [t] at \topt{6mm} \topt{37mm}
    \pinlabel {$ C_{2,1} $} [t] at \topt{110mm} \topt{37mm}
    \pinlabel {$ P_k $} [Bl] at \topt{94mm} \topt{4mm}
    \endlabellist
    \includegraphics{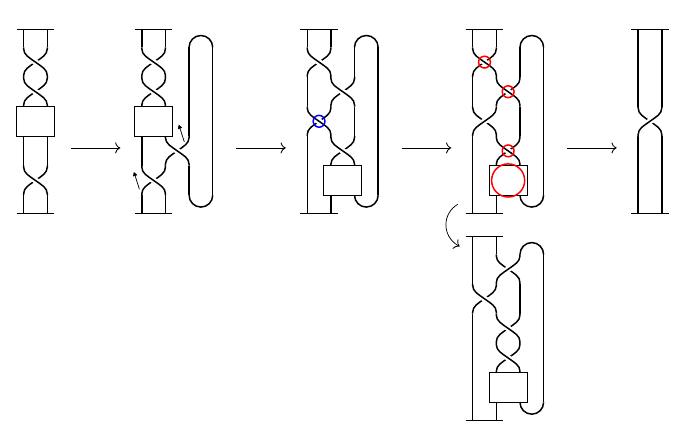}
    \caption{The box labeled ``$ -k $'' represents $ k $ negative full twists for any $ k \geq 0 $. Above is a cobordism $ \Sigma\colon C_{2,-2k-3} \to C_{2, 1} $ with $ \chi(\Sigma) = -2k - 4 $ and $ P_k $ as a slice. The squeezed pattern is $ P_\text{L6a1} $ for $ k = 0 $ and $ P_\text{L8a5} $ for $ k = 1 $.}
  \end{subfigure}
  \begin{subfigure}{\textwidth}
    \centering
    \labellist
    \footnotesize
    \pinlabel {\parbox{10em}{\centering add disk, \\ 3 bands}} [b] at \topt{16mm} \topt{22mm}
    \pinlabel {change \bluecrossing} [B] at \topt{44mm} \topt{23mm}
    \pinlabel {isotopy} [B] at \topt{72mm} \topt{23mm}
    \pinlabel {isotopy} [B] at \topt{100mm} \topt{23mm}
    \pinlabel {$ C_{2,-1} $} [t] at \topt{6mm} \topt{7mm}
    \pinlabel {$ P_\text{L7a2} $} [t] at \topt{30mm} \topt{7mm}
    \pinlabel {$ C_{2,3} $} [t] at \topt{110mm} \topt{7mm}
    \endlabellist
    \includegraphics{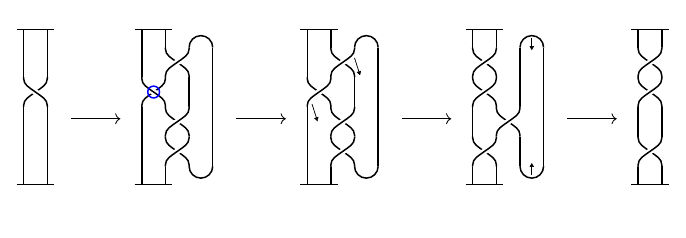}
    \caption{A cobordism $ \Sigma\colon C_{2, -1} \to C_{2, 3} $ with $ \chi(\Sigma) = -4 $ and $ P_\text{L7a2} $ as a slice.}
  \end{subfigure}
  \begin{subfigure}{\textwidth}
    \centering
    \labellist
    \footnotesize
    \pinlabel \parbox{10em}{\centering add disk, \\ 2 bands} [b] at \topt{16mm} \topt{22mm}
    \pinlabel \parbox{10em}{\centering change \bluecrossing, \\ resolve \redcrossing} [b] at \topt{44mm} \topt{22mm}
    \pinlabel {isotopy} [B] at \topt{72mm} \topt{23mm}
    \pinlabel {isotopy} [B] at \topt{100mm} \topt{23mm}
    \pinlabel {$ C_{2,-2} $} [t] at \topt{6mm} \topt{7mm}
    \pinlabel {$ P_\text{L8n1} $} [t] at \topt{30mm} \topt{7mm}
    \pinlabel {$ C_{2,2} $} [t] at \topt{110mm} \topt{7mm}
    \endlabellist
    \includegraphics{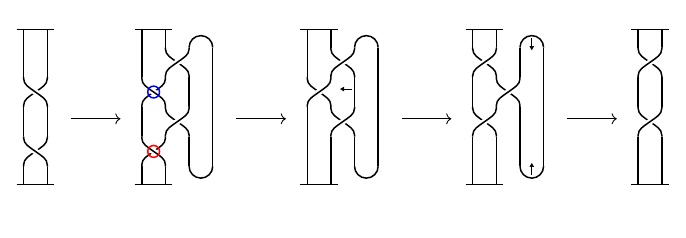}
    \caption{A cobordism $ \Sigma\colon C_{2, -2} \to C_{2, 2} $ with $ \chi(\Sigma) = -4 $ and $ P_\text{L8n1} $ as a slice.}
  \end{subfigure}
  \caption{Squeezing cobordisms for the patterns in Figure~\ref{fig:small-patterns}.}
  \label{fig:small-pattern-cobordisms}
\end{figure}
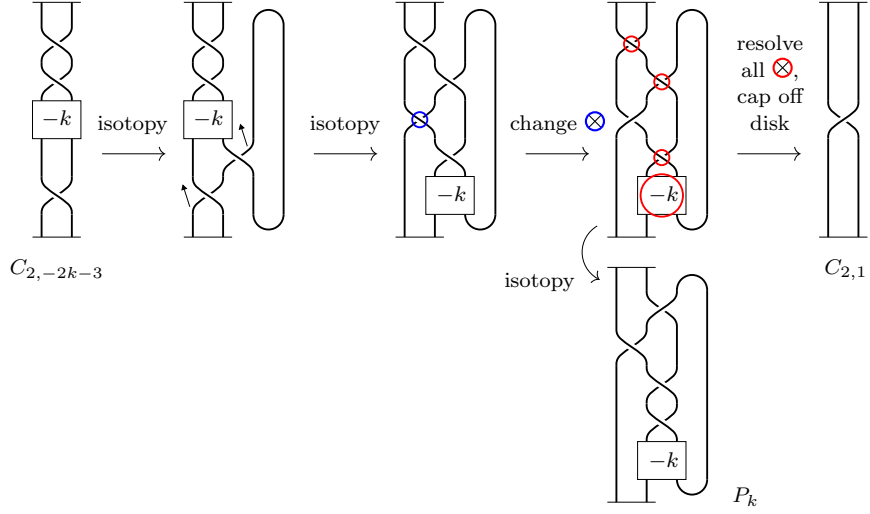
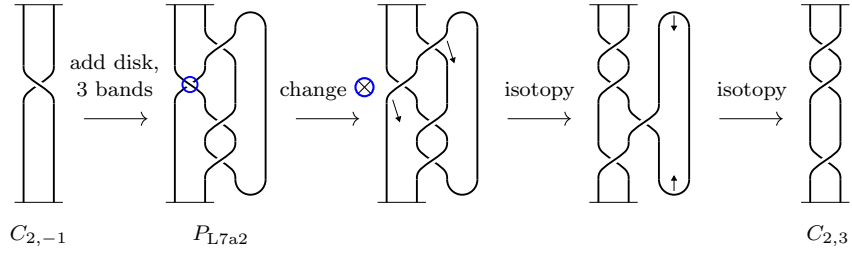
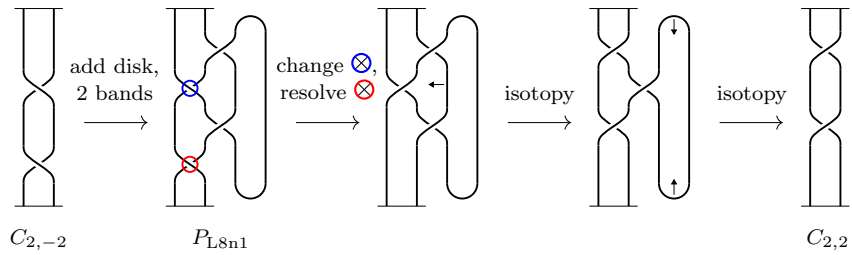

\vfill

\pagebreak

\begin{example}
To construct a non-braided squeezed pattern, we can return to the cobordisms discussed in the proof of Theorem~\ref{thm:tau-of-braided-satellites}, those from $ C_{p, -q} $ to $ C_{p, q} $ constructed by attaching a band resolving each negative crossing and then attaching another band to put a positive crossing in its place. Instead of attaching the resolving band ``straight across,'' we can attach it in a way that is ``locked'' around some number of crossings, as in Figure~\ref{fig:construct-squeezed-pattern}. After changing the sign of those crossings, the locked band comes free and can resolve the crossing as normal. One can construct non-braided squeezed patterns of any winding number this way. The particular pattern constructed in Figure~\ref{fig:construct-squeezed-pattern} is $ P_\text{L9a7} $.

\begin{figure}
  \centering
  \labellist
  \footnotesize
  \pinlabel \parbox{10em}{\centering resolve \redcrossing, \\ attach \\ ``locked'' \\ band} [b] at \topt{24mm} \topt{22mm}
  \pinlabel \parbox{10em}{\centering change \bluecrossing} [b] at \topt{60mm} \topt{21mm}
  \pinlabel \parbox{10em}{\centering isotopy} [b] at \topt{96mm} \topt{21mm}
  \small
  \pinlabel {$ C_{2, -3} $} [B] at \topt{8mm} \topt{2mm}
  \pinlabel {$ P_\text{L9a7} $} [B] at \topt{40mm} \topt{2mm}
  \pinlabel {$ C_{2, 1} $} [B] at \topt{112mm} \topt{2mm}
  \endlabellist
  \includegraphics{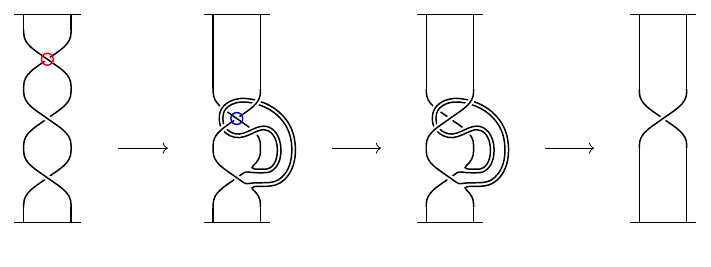}
  \caption{A construction of a non-braided squeezed pattern.}
  \label{fig:construct-squeezed-pattern}
\end{figure}
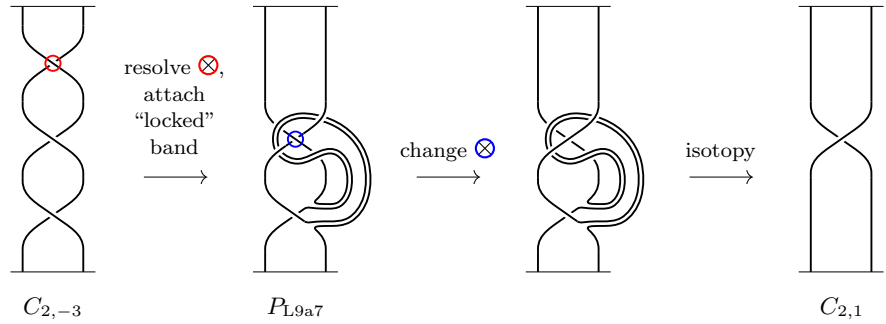
\end{example}

\bibliographystyle{alpha}
\bibliography{library}

\begin{thebibliography}{HRW24}

\bibitem[Bod24]{bodish-genus-fiberedness-t-24-preprint}
Holt Bodish.
\newblock Genus, fiberedness, $\tau$ and $\varepsilon$ of satellite knots with
  $n$-twisted generalized {{Mazur}} patterns.
\newblock Preprint, May 2024.
\newblock \href{http://arxiv.org/abs/2405.08763}{\tt arXiv:2405.08763}.

\bibitem[CH23]{chen-hanselman-satellite-knots-immersed-23-preprint}
Wenzhao Chen and Jonathan Hanselman.
\newblock Satellite knots and immersed {{Heegaard Floer}} homology.
\newblock Preprint, September 2023.
\newblock \href{http://arxiv.org/abs/2309.12297}{\tt arXiv:2309.12297}.

\bibitem[CZZ25]{chen-zemke-zhou-applications-lspace-satellite-25-preprint}
Daren Chen, Ian Zemke, and Hugo Zhou.
\newblock Applications of the {{L-space}} satellite formula.
\newblock Preprint, September 2025.
\newblock \href{http://arxiv.org/abs/2509.20288}{\tt arXiv:2509.20288}.

\bibitem[FLL24]{feller-lewark-lobb-squeezed-knots-24}
Peter Feller, Lukas Lewark, and Andrew Lobb.
\newblock Squeezed knots.
\newblock {\em Quantum Topology}, 16(4):831--865, March 2024.

\bibitem[Hed07]{hedden-knot-floer-homology-07}
Matthew Hedden.
\newblock Knot {{Floer}} homology of {{Whitehead}} doubles.
\newblock {\em Geometry \& Topology}, 11(4):2277--2338, December 2007.

\bibitem[Hom14]{hom-bordered-heegaard-floer-14}
Jennifer Hom.
\newblock Bordered {{Heegaard Floer}} homology and the tau-invariant of cable
  knots.
\newblock {\em Journal of Topology}, 7(2):287--326, June 2014.

\bibitem[HP21]{hedden-pinzon-caicedo-satellites-infinite-rank-21}
Matthew Hedden and Juanita {Pinz{\'o}n-Caicedo}.
\newblock Satellites of infinite rank in the smooth concordance group.
\newblock {\em Inventiones mathematicae}, 225(1):131--157, July 2021.

\bibitem[HR23]{hedden-raoux-4dimensional-rational-genus-23-preprint}
Matthew Hedden and Katherine Raoux.
\newblock A 4-dimensional rational genus bound.
\newblock Preprint, August 2023.
\newblock \href{http://arxiv.org/abs/2308.16853}{\tt arXiv:2308.16853}.

\bibitem[HRW24]{hanselman-rasmussen-watson-bordered-floer-homology-24}
Jonathan Hanselman, Jacob Rasmussen, and Liam Watson.
\newblock Bordered {{Floer}} homology for manifolds with torus boundary via
  immersed curves.
\newblock {\em Journal of the American Mathematical Society}, 37(2):391--498,
  April 2024.

\bibitem[JKM25]{johanningsmeier-kim-miller-partial-resolution-heddens-25}
Randall Johanningsmeier, Hillary Kim, and Allison~N. Miller.
\newblock A partial resolution of {{Hedden}}'s conjecture on satellite
  homomorphisms.
\newblock {\em Mathematical Proceedings of the Cambridge Philosophical
  Society}, 180(1):93--104, August 2025.

\bibitem[Lev16]{levine-nonsurjective-satellite-operators-16}
Adam~Simon Levine.
\newblock Nonsurjective satellite operators and piecewise-linear concordance.
\newblock {\em Forum of Mathematics, Sigma}, 4:e34, December 2016.

\bibitem[Lew14]{lewark-rasmussens-spectral-sequences-14}
Lukas Lewark.
\newblock Rasmussen's spectral sequences and the
  $\mathfrak{sl}_{N}$-concordance invariants.
\newblock {\em Advances in Mathematics}, 260:59--83, August 2014.

\bibitem[Liv04]{livingston-computations-ozsvath-szabo-04}
Charles Livingston.
\newblock Computations of the {{Ozsv\'ath}}--{{Szab\'o}} knot concordance
  invariant.
\newblock {\em Geometry \& Topology}, 8(2):735--742, May 2004.

\bibitem[LM26]{linkinfo}
Charles Livingston and Allison~H. Moore.
\newblock {{LinkInfo}}: {{Table}} of link invariants.
\newblock \url{https://linkinfo.knotinfo.org/}, June 2026.

\bibitem[LMP24]{lidman-miller-pinzon-caicedo-linking-number-obstructions-24}
Tye Lidman, Allison~N. Miller, and Juanita {Pinz{\'o}n-Caicedo}.
\newblock Linking number obstructions to satellite homomorphisms.
\newblock {\em Quantum Topology}, 17(2):431--466, July 2024.

\bibitem[LOT18]{lipshitz-ozsvath-thurston-bordered-heegaard-floer-18}
Robert Lipshitz, Peter Ozsv{\'a}th, and Dylan Thurston.
\newblock Bordered {{Heegaard Floer}} homology: {{Invariance}} and pairing.
\newblock {\em Memoirs of the American Mathematical Society}, 254(1216), July
  2018.

\bibitem[Mil23]{miller-homomorphism-obstructions-satellite-23}
Allison~N. Miller.
\newblock Homomorphism obstructions for satellite maps.
\newblock {\em Transactions of the American Mathematical Society, Series B},
  10(8):220--240, February 2023.

\bibitem[OS03]{ozsvath-szabo-knot-floer-homology-03}
Peter Ozsv{\'a}th and Zolt{\'a}n Szab{\'o}.
\newblock Knot {{Floer}} homology and the four-ball genus.
\newblock {\em Geometry \& Topology}, 7(2):615--639, October 2003.

\bibitem[PX24]{patwardhan-xiao-generalized-mazur-patterns-24-preprint}
Jay Patwardhan and Zheheng Xiao.
\newblock Generalized {{Mazur}} patterns and immersed {{Heegaard Floer}}
  homology.
\newblock Preprint, October 2024.
\newblock \href{http://arxiv.org/abs/2404.14578}{\tt arXiv:2404.14578}.

\bibitem[VC10]{vancott-ozsvathszabo-rasmussen-invariants-10}
Cornelia~A. Van~Cott.
\newblock Ozsvath-{{Szabo}} and {{Rasmussen}} invariants of cable knots.
\newblock {\em Algebraic \& Geometric Topology}, 10(2):825--836, April 2010.

\end{thebibliography}

\end{document}